\documentclass[12pt]{article}

\usepackage{amsmath}
\usepackage{amsthm}
\usepackage{amsfonts}

\usepackage{setspace}

\usepackage{color}

\usepackage{url}

\usepackage{marvosym}
\usepackage{MnSymbol}

\usepackage{verbatimbox} 

\usepackage{tikz-qtree}
\tikzset{downshift/.style={yshift=-40pt}}
\usepackage{cancel}

\newcommand{\avoidindex}{\mathrm{ind}}
\newcommand{\rev}[1]{{#1}^R}
\newcommand{\formula}[1]{\psi_{#1}}
\newcommand{\cycle}[2]{c_{#1}^{({#2})}}

\newtheorem{theorem}{Theorem}[section]

\newtheorem{lemma}[theorem]{Lemma}
\newtheorem{corollary}[theorem]{Corollary}

\theoremstyle{definition}
\newtheorem{definition}{Definition}[section]

\theoremstyle{remark}

\title{A family of formulas with reversal of high avoidability index}
\author{James Currie, Lucas Mol, and Narad Rampersad}
\date{}

\begin{document}

\maketitle

%
%
%
%

\begin{abstract}
We present an infinite family of formulas with reversal whose avoidability index is bounded between $4$ and $5,$ and we show that several members of the family have avoidability index $5.$  This family is particularly interesting due to its size and the simple structure of its members.  For each $k\in\{4,5\},$ there are several previously known avoidable formulas (without reversal) of avoidability index $k,$ but they are small in number and they all have rather complex structure.

\vspace{0.5cm}


\end{abstract}


\section{Preliminaries}

Let $\Sigma$ be a set of letters called \textit{variables}.  A \textit{pattern} $p$ over $\Sigma$ is a finite word over alphabet $\Sigma$.  A \textit{formula} $\phi$ over $\Sigma$ is a finite set of patterns over $\Sigma.$  Instead of set notation, we usually use dot notation to denote formulas; that is, for $p_1,\dots,p_n\in \Sigma^*$ we let
\[
p_1\cdot p_2\cdot \dots\cdot p_n=\{p_1,p_2,\dots,p_n\}.
\]
The elements of a formula $\phi$ are called the \textit{fragments} of $\phi.$

For an alphabet $\Sigma,$ define the \textit{reversed alphabet} $\Sigma^R=\{\rev{x}\colon\ x\in \Sigma\},$ where $\rev{x}$ denotes the \textit{reversal} or \textit{mirror image} of variable $x.$  A \textit{pattern with reversal} is a pattern over $\Sigma\cup\Sigma^R.$  
A \textit{formula with reversal} over $\Sigma$ is a formula over $\Sigma\cup\Sigma^R,$ i.e.\ a finite set of patterns with reversal over $\Sigma.$

For words over any alphabet $A,$ we denote by $-^R$ the reversal antimorphism; if $a_1,a_2,\dots, a_n\in A$, then
\[
(a_1a_2\dots a_n)^R=a_na_{n-1}\dots a_1.
\]
We say that a morphism $f:(\Sigma\cup \Sigma^R)^*\rightarrow A^*$ \textit{respects reversal} if $f(\rev{x})=f(x)^R$ for all variables $x\in \Sigma.$  Note that any morphism $f:\Sigma^*\rightarrow A^*$ extends uniquely to a morphism from $(\Sigma\cup \Sigma^R)^*$ that respects reversal.

Let $p$ be a pattern (with reversal).  An \textit{instance} of $p$ is the image of $p$ under some non-erasing morphism (respecting reversal).  A word $w$ \textit{avoids} $p$ if no factor of $w$ is an instance of $p.$  Let $\phi$ be a binary formula (with reversal).  We say that $\phi$ \textit{occurs} in $w$ if there is a non-erasing morphism $h$ (which respects reversal) such that the $h$-image of every fragment of $\phi$ is a factor of $w.$  In this case we say that $\phi$ occurs in $w$ \textit{through $h$}, or that $w$ \textit{encounters} $\phi$ through $h.$  If $\phi$ does not occur in $w$ then we say that $w$ \textit{avoids} $\phi.$  For a positive integer $k,$ let $A_k$ denote an alphabet on $k$ letters.  We say that formula $\phi$ is \textit{$k$-avoidable} if there are infinitely many words of $A_k^*$ which avoid $\phi$;  or equivalently, if there is an $\omega$-word $\mathbf{w}$ over $A_k$ such that every finite prefix of $\mathbf{w}$ avoids $\phi.$  If $\phi$ is not $k$-avoidable we say that $\phi$ is \textit{$k$-unavoidable}.  We say that $\phi$ is \textit{avoidable} if it is $k$-avoidable for some $k\in\mathbb{N}$; otherwise, we say that $\phi$ is \textit{unavoidable}.  Finally, the \textit{avoidability index} of $\phi,$ denoted $\avoidindex(\phi),$ is the least $k\in\mathbb{N}$ such that $\phi$ is $k$-avoidable if $\phi$ is avoidable, and is $\infty$ if $\phi$ is unavoidable.

An open question in pattern avoidance is whether patterns of arbitrarily high avoidability index exist.  At the time of writing it is unknown whether any patterns of avoidability index strictly greater than $5$ exist.  Formulas have been important to the search for patterns with high avoidability index because of the following connection between patterns and formulas.  For every formula, there is an associated pattern of the same avoidability index obtained by replacing every dot with a new distinct letter (this fact was first proven in \cite{CassaigneThesis}, and a similar argument demonstrates the truth of this fact for formulas with reversal as well).  Indeed, this was the primary reason that formulas were introduced, as they are somewhat easier to analyze than the associated patterns.

While the fact that the pattern $xx$ has avoidability index $3$ is a well-known classical result, the first known pattern of avoidability index $4,$ presented by Baker, McNulty, and Taylor in \cite{Index4} arises from the much longer formula $ab\cdot ba\cdot ac\cdot ca\cdot bc.$  Later, Clark \cite{ClarkThesis} demonstrated that every avoidable formula on at most three variables is $4$-avoidable, but found several formulas on four or more variables that have avoidability index $5.$  All of these formulas are rather long and complex.

Relatively little is known about the possible values of the avoidability index of patterns with reversal.  Currie and Lafrance \cite{BinaryPatternsReversal} found the avoidability index of every \textit{binary} pattern with reversal; that is, every pattern on $\Sigma\cup \Sigma^R$ where $|\Sigma|=2.$  In particular, they found that every avoidable binary pattern with reversal has avoidability index at most $3$.  Currie and Rampersad have shown that the growth of the number of binary words avoiding the formula $x\rev{x}x$ is intermediate between polynomial and exponential \cite{CurrieRampersad2015}, a surprising result which they have also shown to hold for the formula $xx\rev{x}$ \cite{CurrieRampersad2016}.  These are the first known instances of such an intermediate growth rate in the context of pattern avoidance, and this suggests that patterns with reversal may be quite different from patterns in the usual sense (i.e.\ without reversal).

In this article we present an infinite family of formulas with reversal whose avoidability index is bounded between $4$ and $5,$ and we show that several members of the family have avoidability index $5.$  This family is actually part of a larger family of formulas with reversal whose members are shown to have avoidability index between $4$ and $7$ in general (although we suspect that the true upper bound is $5$).  The simplicity of our examples makes their high avoidability index all the more surprising, as the previously known formulas (without reversal) of avoidability index $4$ or $5$ are quite complex.

For each $k\geq 1,$ define
\[
\formula{k}=xy_1y_2\dots y_kx\cdot\rev{y_1}\cdot \rev{y_2}\cdot \dots\cdot\rev{y_k}.
\]
In Section \ref{Formula1and2} we show that $\avoidindex(\formula{1})=4$ and $\avoidindex(\formula{2})=5.$  Then we move on to general bounds on the avoidability index of $\formula{k}$ for $k\geq 3.$  In section \ref{GeneralUpperBound} we bound the avoidability index of $\formula{k}$ from above for all $k\geq 3.$  We show that $\avoidindex(\formula{3k})\leq 5$ for $k\geq 1$, $\avoidindex(\formula{3k+1})\leq 6$ for all $k\geq 1,$ and $\avoidindex(\formula{3k+2})\leq 6$ for all $k\geq 2.$  The only remaining case is $\formula{5},$ for which we demonstrate $\avoidindex(\formula{5})\leq 7.$  While we can use backtracking to show that $\avoidindex(\formula{k})\geq 5$ for all $3\leq k\leq 6$ (so in particular $\avoidindex(\formula{3})=\avoidindex(\formula{6})=5$), this method becomes impractical for larger values of $k.$  In Section \ref{GeneralLowerBound} we present a general argument that shows $\avoidindex(\formula{k})\geq 4$ for all $k\geq 7.$

\section{The avoidability index of $\formula{1}$ and $\formula{2}$}\label{Formula1and2}

This section is devoted to proving that $\avoidindex(\formula{1})=4$ and $\avoidindex(\formula{2})=5$.  First of all, one can demonstrate that $\formula{1}$ is $3$-unavoidable and that $\formula{2}$ is $4$-unavoidable by using a standard backtracking algorithm.  It remains to show that $\formula{1}$ is $4$-avoidable and $\formula{2}$ is $5$-avoidable.  In fact, we show that there are exponentially many words on $4$ letters that avoid $\formula{1}$ and exponentially many words on $5$ letters that avoid $\formula{2}.$  Below we define words that have a particular cyclic structure as they will be used in the constructions that follow.

\begin{definition}
Let $k\geq 1$ and let $w=w_1w_2\dots$ be a word in $\{1,\dots,k\}^*.$  Define the \textit{$(a_1,\dots ,a_m)$-cyclic $w$-word} $\cycle{m}{w}$ on distinct letters $a_1,\dots, a_m$ by
\[
\cycle{m}{w}=a_1^{w_1}a_2^{w_2}\dots a_{m}^{w_m}a_1^{w_{m+1}}a_2^{w_{m+2}}\dots a_m^{w_{2m}}\dots.
\] 
Any word isomorphic to $\cycle{m}{w}$ is called an \textit{$m$-cyclic $w$-word}.
\end{definition}

We will find infinitely many words for which the corresponding $4$-cyclic word  avoids $\formula{1},$ and infinitely many words for which the corresponding $5$-cyclic word avoids $\formula{2}.$  Further, we will show that the growth of the number of these words is exponential in each case.  We begin with a lemma that gives a condition on a word $w\in\{1,\dots,k+1\}$ that is satisfied if and only if $\cycle{m}{w}$ avoids $\formula{k}$ (with $m\geq k+2$).  Hence this lemma allows us to determine whether a given $m$-cyclic $w$-word avoids $\formula{k}$ by considering $w$ alone.

\begin{lemma}\label{GeneralLemma}
Let $k\geq 1$ and $m\geq k+2.$  If $w\in\{1,\dots,k+1\}^*$ then the $m$-cyclic $w$-word $\cycle{m}{w}$ avoids $\formula{k}$ if and only if for all $j\in\{1,\dots,k\},$ $w$ has no factor of the form $x'\alpha_1\alpha_2\dots\alpha_j x'',$ where 
\begin{itemize}
\item $|\alpha_i|=1$ for all $i\in\{1,\dots,j\},$ i.e.\ $\alpha_i\in \{1,\dots,k+1\}$; 
\item $\displaystyle\sum_{i=1}^j \alpha_i\geq k$ (in $\mathbb{Z}$);
\item $|x'|=|x''|=n$ with $n\equiv m-j \pmod{m}$; and
\item if $x'=x'_1\dots x'_n$ and $x''=x''_1\dots x''_n$ then $x'_1\geq x''_1,$ $x'_n\leq x''_n,$ and $x'_i=x''_i$ for all $i\in\{2,\dots,n-2\}.$
\end{itemize}
\end{lemma}

\begin{proof}
First suppose that $w$ has a factor $u$ of the form $x'\alpha_1\dots\alpha_j x''$ for some $j\in\{1,\dots,k\}$ satisfying the conditions of the lemma statement.  We have
\[
\cycle{m}{u}=a_1^{x'_1}a_2^{x'_2}\dots a_{m-j}^{x'_n}a_{m-j+1}^{\alpha_1}\dots a_m^{\alpha_j}a_1^{x''_1}a_2^{x''_2}\dots a_{m-j}^{x''_n},
\]
and this word occurs as a factor of $\cycle{m}{w}$ (up to a shifting of the letters $a_1,\dots a_m$).  For each $i\in\{2,\hdots,n-1\},$ we have $x'_i=x''_i,$ so define $x_i=x'_i=x''_i.$  Notice that the factor
\[
a_1^{x''_1}a_2^{x_2}\dots a_{m-j-1}^{x_{n-1}} a_{m-j}^{x'_n}
\]
appears on either side of $a_{m-j+1}^{\alpha_1}\dots a_m^{\alpha_j}$ in $\cycle{m}{u},$ since $x''_1\leq x'_1$ and $x'_n\leq x''_n.$  Further, the facts that $j\leq k$ and $\displaystyle\sum_{i=1}^j \alpha_i\geq k$ (in $\mathbb{Z}$) allow us to factor $a_{m-j+1}^{\alpha_1}\dots a_m^{\alpha_j}$ into $k$ nonempty words each of which is a power of a single letter.  We conclude that the formula $\formula{k}$ occurs in $\cycle{m}{u},$ and hence also occurs in $\cycle{m}{w}.$

For the converse, suppose that $\formula{k}$ occurs in $\cycle{m}{w}$ through the morphism $h.$  First note that $h(xy_1\dots y_k x)$ cannot be a power of a single letter, as the powers in $\cycle{m}{w}$ come from $w\in\{1,\dots, k+1\}^*.$  However, for each $i\in\{1,\dots,k\},$ $h(y_i)$ must be a power of a single letter, as these are the only reversible factors of $\cycle{m}{w}$ (since $m\geq 3$).  Suppose without loss of generality that $h(y_1)=a_1^{p}$ and $h(y_k)=a_j^{q}.$  We necessarily have $j\leq k$ due to the cyclic ordering of letters in $\cycle{m}{w}.$  We also see that we must have $h(x)=a_{j}^r v a_1^s,$ for some $r,s\geq 0,$ and
\[
v=a_{j+1}^{t_1}a_{j+2}^{t_2}\dots a_m^{t_n}
\]  
with $n\equiv m-j\pmod{m}$, $t_i\in\{1,\dots,k+1\}$ for all $i\in\{1,\dots,n\}.$  Define morphism $h'$ by $h'(x)=v,$ $h'(y_1)=a_{1}^sh(y_1),$ $h'(y_k)=h(y_k)a_{j}^r,$ and $h'(y_i)=h(y_i)$ for all $i\in\{2,\hdots,k-1\}.$  Then $h'$ is also an occurrence of $\formula{k}$ in $\cycle{m}{w}.$  

Now if we write $h'(y_1\dots y_k)=a_1^{\alpha_1}a_2^{\alpha_2}\dots a_j^{\alpha_j},$ then each instance of the factor $h'(xy_1\dots y_kx)=vh'(y_1\dots y_k)v$ in $\cycle{m}{w}$ corresponds to a factor
\[
u=t_1^+t_2t_3\dots t_n \alpha_1\alpha_2\dots \alpha_j t_1t_2\dots t_n^+
\]
of $w,$ where $t_1^+\geq t_1$ and $t_n^+\geq t_n$ (since we don't know what happens to the left and right of an instance of $h'(xy_1\dots y_kx)$ in $\cycle{m}{w}$).  The factor $u$ clearly has the form $x'\alpha_1\dots \alpha_j x''$ described in the lemma statement.
\end{proof}

We are now ready to prove that $\formula{1}$ is $4$-avoidable.  The proof relies on an $8$-uniform morphism  $f$ that takes any square-free ternary word $v$ to a binary word $f(v)$ for which the associated $4$-cyclic word avoids $\formula{1}.$  From here, the exponential growth of the number of words avoiding $\formula{1}$ is straightforward to prove using the fact that $v$ can be any square-free ternary word.

\begin{theorem}\label{Formula1Theorem}
The formula $\formula{1}$ is $4$-avoidable.  Further, the growth of the number of words on $4$ letters avoiding $\formula{1}$ is exponential.
\end{theorem}

\begin{proof}
Define a morphism $f:\{0,1,2\}^*\rightarrow \{1,2\}^*$ by 
\begin{align*} 
f(0)&=11112122,\\
f(1)&=12112222, \mbox{ and}\\
f(2)&=21111222.
\end{align*}
We claim that if $v$ is a square-free ternary word, then $f(v)$ has no factor of the form $x'\alpha x'',$ where $|\alpha|=1,$ $|x'|=|x''|=n$ with $n\equiv{3}\pmod{4},$ and $x'=x''$ except possibly at the first and last letters (and thus $\cycle{4}{f(v)}$ avoids $\formula{1}$ by Lemma \ref{GeneralLemma}).  We prove the contrapositive of this claim below.

Let $w=f(v)=w_0w_1w_2\dots$ and suppose that the factor 
\[
w_{k-n}\dots w_{k-1}w_kw_{k+1}\dots w_{k+n}
\] 
has the form $x'\alpha x''$ described above.  In particular, $n\equiv 3\pmod{4}$, $w_{k-n}\geq w_{k+1},$ $w_{k-1}\leq w_{k+n}$, and $w_{k-n+i}=w_{k+1+i}$ for all $i\in\{1,\dots,n-2\}.$    We have verified that no factor of this form occurs in the image of any square-free ternary word of length $2,$ so we may assume that $n>3.$  In what follows, we refer to the image of a single letter of $v$ as a \textit{code word} in $w$.

We first demonstrate that if $w_{k-n}\dots w_{k+n}$ has the form $x'\alpha x''$ described above, then $n\equiv 7 \pmod{8}.$  By assumption, $n\equiv 3\pmod{4},$ so it suffices to show that $n\not\equiv 3\pmod{8}.$  Suppose towards a contradiction that $n\equiv 3\pmod{8}.$  Using the assumption that $n>3,$ in particular we have
\begin{align}
w_{k-n+1}w_{k-n+2}w_{k-n+3}=w_{k+2}w_{k+3}w_{k+4}.\label{LemmaEquation1}
\end{align}
Note that the third and fourth letters of every code word are $1$ while the seventh and eighth letters are $2.$  Some letter from $w_{k-n+1}w_{k-n+2}w_{k-n+3}$ must sit in either the third, fourth, seventh, or eighth position of a code word.  However, since $n+1\equiv 4\pmod{8}$ the letter in the corresponding position in $w_{k+2}w_{k+3}w_{k+4}$ has the opposite identity.  This contradicts (\ref{LemmaEquation1}).

Now we may assume that $n\equiv 7\pmod{8}.$  A key observation used below is that if two code words match at their first and second letters, then they are equal.  This follows directly from the fact that $f(0),$ $f(1),$ and $f(2)$ all have distinct prefixes of length $2.$  Similarly, if two code words match at their fifth and sixth letters, then they are equal.  Finally, since $f$ is $8$-uniform we can tell exactly where code words begin and end in $w_{k-n}\dots w_{k+n}$ from the value of $k\bmod{8}.$

\begin{description}
\item[Case I:] If $k\equiv 0\pmod{8}$ then the code words
\[
w_{k-n-1}\dots w_{k-n+6} \mbox{ and } w_k\dots w_{k+7}
\] 
are equal since they match at their fifth and sixth letters.  If $n=7$ we are done, as we have two identical code words in a row, which must have come from a square (of length $2$) in $w.$  Otherwise, the code words
\[
w_{k-8}\dots w_{k-1} \mbox{ and } w_{k+n-7}\dots w_{k+n}
\] 
are also equal since they match at their fifth and sixth letters.  Altogether we have 
\[
w_{k-n-1}\dots w_{k-1}=w_k\dots w_{k+n},
\] 
and hence the preimage of $w_{k-n-1}\dots w_{k+n}$ in $f$ is a square.  The same argument works when $k\equiv 7\pmod{8}$ with all indices shifted to the right by $1.$

\item[Case II:] If $k=1\pmod{8}$ then the code words 
\[
w_{k-n-2}\dots w_{k-n+5} \mbox{ and } w_{k-1}\dots w_{k+6}
\]
are equal since they match at their fifth and sixth letters.  Thus we have 
\[
w_{k-n-2}\dots w_{k-2}=w_{k-1}\dots w_{k+n-1},
\] 
and hence the preimage of $w_{k-n-2}\dots w_{k+n-1}$ in $f$ is a square.  The same argument works when $k=2\pmod{8}$ with all indices shifted to the left by $1.$

\item[Case III:] If $k\equiv 3\pmod{8},$ then the code words
\[
w_{k-3}\dots w_{k+4} \mbox{ and } w_{k+n-2}\dots w_{k+n+5}
\]
are equal since they match at their first and second letters.  Thus we have 
\[
w_{k-n+4}\dots w_{k+4}=w_{k+5}\dots w_{k+n+5},
\] 
and hence the preimage of $w_{k-n+4}\dots w_{k+n+5}$ in $f$ is a square.  The same argument works when $k\equiv 4,\mbox{ }5,\mbox{ and } 6\pmod{8}$ with all indices shifted to the left by $1$, $2$, and $3,$ respectively.
\end{description}

Therefore, if $v\in\{0,1,2\}^*$ is square-free, then the binary word $f(v)$ avoids factors of the form $x'\alpha x''.$  By Lemma \ref{GeneralLemma}, the $4$-cyclic $f(v)$-word $\cycle{m}{f(v)}$ avoids $\formula{1}.$  Since there are infinitely many square-free ternary words, we conclude that $\formula{1}$ is $4$ avoidable.  It remains to show that the growth of the number of words on $4$ letters avoiding $\formula{1}$ is exponential.

It is well known that the number of square-free ternary words grows exponentially \cite{Brandenburg}.  However, this does not immediately imply that the number of words on $4$ letters avoiding $\formula{1}$ grows exponentially, because there are square-free words $u$ and $v$ in $\{0,1,2\}^*$ of the same length for which $\cycle{4}{f(u)}$ and $\cycle{4}{f(v)}$ have different lengths.  We would like $f(u)$ and $f(v)$ to have the same number of $1$'s and $2$'s so that we have $|\cycle{4}{f(u)}|=|\cycle{4}{f(v)}|.$  One way to ensure this is to start with words $u$ and $v$ that have the letters $0$, $1$, and $2$ in the same proportions.  Fortunately, the number of square-free ternary words in which each alphabet letter occurs with proportion exactly $1/3$ grows exponentially (while not explicitly stated there, this fact can easily be gleaned from Section 4.1 of \cite{RichardGrimm2004}; apply any of the Brinkhuis triples given there starting from initial word $abc$).  From this fact we may conclude that the number of words on $4$ letters avoiding $\formula{1}$ grows exponentially.
\end{proof}

Next we show that $\formula{2}$ is $5$-avoidable.  The proof is much shorter than the preceding proof that $\formula{1}$ is $4$-avoidable because we make use of automatic theorem-proving software.  We show that there is an infinite binary word $\mathbf{w}$ such that the $5$-cyclic $\mathbf{w}$-word $\cycle{5}{\mathbf{w}}$ avoids $\formula{2}.$

\begin{theorem}\label{Formula2Theorem}
The formula $\formula{2}$ is $5$-avoidable.
\end{theorem}

\begin{proof}
Define morphism $\rho:\{1,2\}^*\rightarrow \{1,2\}^*$ by $\rho(1)=22$ and $\rho(2)=21.$  Let 
\[
\mathbf{w}=\rho^\infty(2)=21222121\dots.
\]  
We claim that $\cycle{5}{\mathbf{w}}$ avoids $\formula{2}.$  By Lemma \ref{GeneralLemma}, it is sufficient to show that $w$ has no factor of the form $x'\alpha_1x''$ or $x'\alpha_1\alpha_2x''$ satisfying the conditions listed in Lemma \ref{GeneralLemma}.  In order to do this, we use the automatic theorem-proving software \textit{Walnut} \cite{Walnut}.  A description of this method including the particular predicates used is included in Appendix \ref{WalnutAppendix}.
\end{proof}

The fact that the number of words avoiding $\formula{2}$ grows exponentially follows fairly easily from Theorem \ref{Formula2Theorem}.  Essentially, we can replace a fixed number of $2$'s in $\rho^k(2)$ by $3$'s, and this change does not introduce any occurrence of $\formula{2}$ to the corresponding $5$-cyclic word.  This gives us a set of words avoiding $\formula{2}$ which we show grows exponentially in size.

\begin{corollary}
The growth of the number of words on $5$ letters avoiding $\formula{2}$ is exponential.
\end{corollary}

\begin{proof}
From the proof of Theorem \ref{Formula2Theorem}, the infinite word
\[
\rho^\infty(2)=21222121\dots
\]
has no factor of the form $x'\alpha_1x''$ or $x'\alpha_1\alpha_2x''$ satisfying the conditions listed in Lemma \ref{GeneralLemma}.  Consider the finite prefix $\rho^k(2)$ for some $k\geq 2$ and let $n=2^k=|\rho^k(2)|.$  Note that $|\rho^k(2)|_2\geq \tfrac{n}{2}.$  Note also that any word obtained from $\rho^k(2)$ by changing any number of $2$'s to $3$'s still has no factors of the form $x'\alpha_1x''$ or $x'\alpha_1\alpha_2x''.$

Let $U_k$ be the set containing all words obtained from $\rho^k(2)$ by replacing exactly $\tfrac{n}{4}$ of the $2$'s by $3$'s.  Clearly no word in $U_k$ has a factor of the form $x'\alpha_1x''$ or $x'\alpha_1\alpha_2x'',$ as this would imply that $\rho^k(2)$ has such a factor.  Thus for any word $u\in U_k,$ the $5$-cyclic $w$-word $\cycle{5}{u}$ avoids $\formula{2}$ by Lemma \ref{GeneralLemma}.  Further, any two words $u,v\in W_k$ have $|u|_i=|v|_i$ for all $i\in\{1,2,3\},$ so that every word in the set
\[
\{\cycle{5}{u}\colon\ u\in U_k\}
\]
has the same length, which is at most $3n$.  Since there are at least $\binom{n/2}{n/4}$ words in $U_k$, we conclude that the growth of the number of words on $5$ letters avoiding $\formula{2}$ is exponential.
\end{proof}

After discovering that $\avoidindex(\formula{1})=4$ and $\avoidindex(\formula{2})=5,$ we were led to wonder whether $\avoidindex(\formula{k})$ grows indefinitely along with $k.$  We provide a negative answer to this question in the next section.

\section{An upper bound on $\avoidindex(\formula{k})$ for $k\geq 3$}\label{GeneralUpperBound}

Our main result in this section is a general construction that shows the $5$-avoidability of $\formula{3k}$ for all $k\geq 1.$  This leads us to believe that $\formula{k}$ is $5$-avoidable for all $k \geq 3.$  Although we are unable to verify this conjecture, we adapt our construction for $\formula{3k}$ to show that $\formula{3k+1}$ is $6$-avoidable for $k\geq 1$ and $\formula{3k+2}$ is $6$-avoidable for $k\geq 2.$  Finally, we address the only remaining formula $\formula{5},$ showing that it is $7$-avoidable.

We start with the main result concerning $\formula{3k}.$  The proof makes use of some new terminology for ease of reading.  A \textit{reversed variable} in a formula $\phi$ is a variable $z$ such that $z$ and $\rev{z}$ both appear in $\phi.$  In particular, in $\formula{k}$ the reversed variables are $y_1,\dots, y_{k}.$

\begin{theorem}
For all $k\geq 1,$ the formula $\formula{3k}$ is $5$-avoidable.  Further, the number of words on $5$ letters avoiding $\formula{3k}$ grows exponentially.
\end{theorem}

\begin{proof}
Fix $k\geq 1$ and let $w=w_0w_1\dots$ be a word in $\{0,1,2\}^*.$  Define the morphism $d_{k}:\{0,1,2\}^*\rightarrow\{0,1,2\}^*$ by $i\mapsto i^{k+1}$ for all $i\in\{0,1,2\}$ and the morphism $g:\{0,1,2\}^*\rightarrow \{0,1,2,a,b\}^*$ by $i\mapsto iab$ for all $i\in\{0,1,2\}.$  We claim that if $\formula{3k}$ occurs in $g(d_k(w))$, then $w$ contains a square.

Suppose that $\formula{3k}$ occurs in $g(d_k(w))$ through morphism $h.$  Several observations can be made from the fact that $g(d_k(w))$ alternates between a letter from $\{0,1,2\}$ and the factor $ab$:

\begin{itemize}

\item the $h$-image of each reversed variable in $\formula{3k}$ must have length $1,$ as these are the only reversible factors of $g(d_k(w))$;

\item exactly $k$ of the $3k$ reversed variables have $h$-image in $\{0,1,2\}$; and

\item $|h(x)|=0 \pmod{3},$ hence $h(x)$ contains at least one letter from $\{0,1,2\}.$

\end{itemize}

Let $u$ be the word obtained by dropping all $a$'s and $b$'s from $h(x)$ and let $v$ be the word obtained by dropping all $a$'s and $b$'s from $h(y_1\dots y_{3k})$.  Clearly $uvu$ is a factor of $d_k(w)$ as we are essentially taking a preimage in $g,$ and by the remarks above we have $|u|\geq 1$ and $|v|=k.$   Up to relabelling of letters, we have either $v=0^k$ or $v=0^i1^j$ with $i,j$ positive and $i+j=k.$  We show below that in each of these cases the factor $uvu$ extends to a factor in $d_k(w)$ that can only have come from a square in $w$.
\begin{description}
\item[Case I:] $v=0^k$

If the factor $00$ occurs in $w,$ then we are done.  So we may assume that $0^{k+2}$ is not a factor of $uvu.$  Hence $u$ must either start with a single $0$ and end with a different letter or end with a single $0$ and start with a different letter.  But then $uvu$ appears internally as  ${\mid}0^ku{\mid}0^ku{\mid}$ or ${\mid}u0^k{\mid}u0^k{\mid},$ respectively.  The preimage in $d_k$ of each of these factors is a square in $w.$

\item[Case II:] $v=0^i1^j$ for positive $i,j$ with $i+j=k$

In this case, the factor $uvu$ in $d_k(w)$ always appears inside the factor
\[
{\mid}1^ju0^i{\mid}1^ju0^i{\mid},
\]
whose preimage is a square in $w.$
\end{description}

We conclude that if $w$ is square-free, then $g(d_k(w))$ avoids the formula $\formula{3k}.$  Since the growth of the number of square-free words on $\{0,1,2\}^*$ is exponential \cite{Brandenburg}, we conclude that the number of words avoiding $\formula{3k}$ grows exponentially as well, as $|g(d_k(w))|=(3k+3)|w|,$ i.e.\ $g(d_k(w))$ is only a constant factor longer than $w.$
\end{proof}

Now we obtain an upper bound on $\avoidindex(\formula{3k+1})$ and $\avoidindex(\formula{3k+2})$ using a similar idea to the one used above to show $\avoidindex(\formula{3k})\leq 5$.  It is easily verified that any long enough word of the form $g(d_k(w))$ encounters both $\formula{3k+1}$ and $\formula{3k+2}$ whether $w$ is square-free or not, so the exact construction used for $\formula{3k}$ will not work for $\formula{3k_1}$ or $\formula{3k+2}.$  However, it turns out that we only need to make a slight modification to $g(d_k(w))$ to make it avoid $\formula{3k+1}$ (or $\formula{3k+2}$) whenever $w$ is square-free.  All we need to do is add a new letter $c$ to $g(d_k(w))$ at certain carefully chosen locations.  The details are given in the corollaries presented below.

\begin{corollary}\label{3k+1}
For all $k\geq 1,$ the formula $\formula{3k+1}$ is $6$-avoidable and the number of words on $6$ letters avoiding $\formula{3k+1}$ grows exponentially.
\end{corollary}

\begin{proof}
Let $w$ be a word in $\{0,1,2\}^*$ and let $d_k$ and $g$ be as in Theorem \ref{GeneralUpperBound}.  From $g(d_k(w)),$ create a word $u_w$ on alphabet $\{0,1,2,a,b,c\}^*$ by inserting the letter $c$ after every $3k$ letters of $g(d_k(w)).$  Now if $\formula{3k+1}$ occurs in $u_w$ through morphism $h,$ it is easily verified that the following conditions hold:
\begin{itemize}
\item the $h$-image of each reversed variable has length $1$;
\item exactly $k$ of the $3k+1$ reversed variables have $h$-image in $\{0,1,2\}$; and
\item $h(x)$ contains at least one letter from $\{0,1,2\}.$
\end{itemize}
Therefore, by arguments very similar to those used in Theorem \ref{GeneralUpperBound}, if $w$ is square-free then $u_w$ avoids $\formula{3k+1},$ and it follows that there are exponentially many words on $6$ letters avoiding $\formula{3k+1}.$
\end{proof}

\begin{corollary}\label{3k+2}
For all $k\geq 2,$ the formula $\formula{3k+2}$ is $6$-avoidable and the number of words on $6$ letters avoiding $\formula{3k+2}$ grows exponentially.
\end{corollary}

\begin{proof}
Let $w$ be a word in $\{0,1,2\}^*$ and let $d_k$ and $g$ be as in Theorem \ref{GeneralUpperBound}.  From $g(d_k(w)),$ create a word $v_w$ on alphabet $\{0,1,2,a,b,c\}^*$ by inserting the letter $c$ after an appearance of the letter $b$ whenever the total number of $b$'s that have occurred so far is equivalent to $0$ or $1$ modulo $k.$  Then in any factor of length $3k+2$ of $v_w$ we have exactly $k$ appearances of letters from $\{0,1,2\},$ $k$ appearances of $a,$ $k$ appearances of $b,$ and two appearances of $c.$  The remainder of the proof is analogous to that of Corollary \ref{3k+1}, and is omitted.
\end{proof}

The construction of Corollary \ref{3k+2} does not work for $k=1$ (that is, for $\formula{5}$).  We show below that $\avoidindex(\formula{5})\leq 7$ using a construction similar to the ones we have already seen.  This does not seem optimal, but we have not found a construction using fewer letters.

\begin{corollary}
The formula $\formula{5}$ is $7$-avoidable and the number of words on $7$ letters avoiding $\formula{5}$ grows exponentially.
\end{corollary}

\begin{proof}
Let $w$ be a word in $\{0,1,2\}^*$ and let $d_2$ be as in Theorem \ref{GeneralUpperBound}.  Define morphism $g':\{0,1,2\}^*\rightarrow\{0,1,2,a,b,c,d\}^*$ by $i\mapsto iabcd$ for all $i\in\{0,1,2\}.$  By arguments similar to those already seen in Theorem \ref{GeneralUpperBound}, if $w$ is square-free then $g'(d_2(w))$ avoids $\formula{5},$ and it follows that there are exponentially many words on $7$ letters avoiding $\formula{5}.$
\end{proof}

Now that we have an upper bound on $\avoidindex(\formula{k})$ for all $k\geq 3,$ we prove a nontrivial lower bound on $\avoidindex(\formula{k})$ in the next section.

\section{A lower bound on $\avoidindex(\formula{k})$ for $k\geq 3$}\label{GeneralLowerBound}

It is trivially true that $\avoidindex(\formula{k})\geq 2$ for all $k\geq 1$ since we have proven that $\formula{k}$ is avoidable and it is obvious that $\formula{k}$ is $1$-unavoidable (this is of course true for every formula with reversal).  While backtracking shows that $\formula{k}$ is $4$-unavoidable for $3\leq k\leq 6,$ this calculation gets more computationally intensive as the value of $k$ grows, eventually becoming infeasible.  The main result of this section is that $\avoidindex(\formula{k})\geq 3$ for the remaining cases $k\geq 7.$  We begin with a lemma that will be required to prove this main result.

\begin{lemma}\label{3CyclicUnavoidable}
For any $k,n\geq 1,$ the $3$-cyclic $w$-word $\cycle{3}{w}$ encounters $\formula{k}$ for any word $w\in\{1,\dots,n\}^\omega.$
\end{lemma}

\begin{proof}
Let $w=w_0w_1\dots$ and let $\cycle{3}{w}=0^{w_0}1^{w_1}2^{w_2}0^{w_3}\dots$ in this proof for ease of reading.  We have three cases, one for each possible value of $k\mod 3.$

\begin{description}
\item[Case I:] $k\equiv 2\pmod{3}$

In this case the formula $\formula{k}$ occurs in $\cycle{3}{w}$ as follows:
\[
\underbrace{0}_{\overset{\upmapsto}{x}}\underbrace{1^{w_1}}_{\overset{\upmapsto}{y_1}}\dots \underbrace{2^{w_k}}_{\overset{\upmapsto}{y_k}} \underbrace{0}_{\overset{\upmapsto}{x}}
\]

\item[Case II:] $k\equiv 0\pmod{3}$

First of all, if $w=1^\omega$ then $\formula{k}$ occurs in $\cycle{3}{w}$ as follows:
\[
\underbrace{012}_{\overset{\upmapsto}{x}}\underbrace{0}_{\overset{\upmapsto}{y_1}}\underbrace{1}_{\overset{\upmapsto}{y_2}}\dots \underbrace{2}_{\overset{\upmapsto}{y_k}} \underbrace{012}_{\overset{\upmapsto}{x}}
\]
We may now assume that some letter of $w$ is greater than $1.$  By shifting the index if necessary, we may assume that $w_1>1.$  Then $\formula{k}$ occurs in $\cycle{3}{w}$ as follows:
\[
\underbrace{0}_{\overset{\upmapsto}{x}}
\underbrace{1}_{\overset{\upmapsto}{y_1}}
\underbrace{1^{w_1-1}}_{\overset{\upmapsto}{y_2}}
\underbrace{2^{w_2}}_{\overset{\upmapsto}{y_3}}
\underbrace{0^{w_3}}_{\overset{\upmapsto}{y_4}}
\dots \underbrace{2^{w_{k-1}}}_{\overset{\upmapsto}{y_k}} 
\underbrace{0}_{\overset{\upmapsto}{x}}
\]

\item[Case III:] $k\equiv 1\pmod{3}$

We have already demonstrated by backtracking that $\formula{1}$ is $3$-unavoidable, so certainly it occurs in every infinite word of the form $\cycle{3}{w}.$  We now handle $k\geq 4.$  First of all, if $w=1^\omega$ then $\formula{k}$ occurs in $\cycle{3}{w}$ as follows:
\[
\underbrace{01}_{\overset{\upmapsto}{x}}
\underbrace{2}_{\overset{\upmapsto}{y_1}}
\underbrace{0}_{\overset{\upmapsto}{y_2}}\dots 
\underbrace{2}_{\overset{\upmapsto}{y_k}} \underbrace{01}_{\overset{\upmapsto}{x}}
\]

Not it suffices to show that $\formula{k}$ occurs in $\cycle{3}{w}$ whenever some letter of $w$ is at least $2.$  First suppose that some letter of $w$ is strictly greater than $2.$  By shifting the index if necessary, we can assume $w_1>2.$  Then $\formula{k}$ occurs in $\cycle{3}{w}$ as follows:
\[
\underbrace{0}_{\overset{\upmapsto}{x}}
\underbrace{1}_{\overset{\upmapsto}{y_1}}
\underbrace{1}_{\overset{\upmapsto}{y_2}}
\underbrace{1^{w_1-2}}_{\overset{\upmapsto}{y_3}}
\underbrace{2^{w_2}}_{\overset{\upmapsto}{y_4}}
\underbrace{0^{w_3}}_{\overset{\upmapsto}{y_5}}
\dots \underbrace{2^{w_{k-2}}}_{\overset{\upmapsto}{y_k}} 
\underbrace{0}_{\overset{\upmapsto}{x}}
\]

Suppose now that $w$ contains the factor $22.$  By shifting the index if necessary we may assume $w_1=w_2=2.$  Then $\formula{k}$ occurs in $\cycle{3}{w}$ as follows:
\[
\underbrace{0}_{\overset{\upmapsto}{x}}
\underbrace{1}_{\overset{\upmapsto}{y_1}}
\underbrace{1^{w_1-1}}_{\overset{\upmapsto}{y_2}}
\underbrace{2}_{\overset{\upmapsto}{y_3}}
\underbrace{2^{w_2-1}}_{\overset{\upmapsto}{y_4}}
\underbrace{0^{w_3}}_{\overset{\upmapsto}{y_5}}
\underbrace{1^{w_4}}_{\overset{\upmapsto}{y_6}}
\dots \underbrace{2^{w_{k-2}}}_{\overset{\upmapsto}{y_k}} 
\underbrace{0}_{\overset{\upmapsto}{x}}
\]

Now suppose that $w$ contains the letter $2$; by shifting the index if necessary we may assume $w_0=2.$  Further, we can take $w_1=1$ since we have shown that any $w$ with the factor $22$ encounters $\formula{k}$.  Finally, since we have shown that any $w$ with a letter strictly greater than $2$ encounters $\formula{k},$ we can take $w_{k+2}\leq 2=w_0.$  Thus, we see that $\formula{k}$ occurs in $\cycle{m}{w}$ as follows:
\[
\underbrace{0^{w_{k+2}}1}_{\overset{\upmapsto}{x}}
\underbrace{2^{w_2}}_{\overset{\upmapsto}{y_1}}
\underbrace{0^{w_{3}}}_{\overset{\upmapsto}{y_2}}\dots 
\underbrace{2^{w_{k+1}}}_{\overset{\upmapsto}{y_k}} 
\underbrace{0^{w_{k+2}}1}_{\overset{\upmapsto}{x}}
\]
\end{description}
We conclude that $\cycle{3}{w}$ encounters $\formula{k}$ for every word $w\in\{1,\dots,n\}^\omega.$
\end{proof}

In other words, Lemma \ref{3CyclicUnavoidable} demonstrates that any $\omega$-word on $3$ letters avoiding $\formula{k}$ has at least one reversible factor containing two distinct letters.  This fact will be important to the proof of the following theorem, the main result of this section.

\begin{theorem}\label{3unavoidable}
For all $k\geq 1,$ the formula $\formula{k}$ is $3$-unavoidable.
\end{theorem}

\begin{proof}
We have shown by backtracking that $\formula{k}$ is $3$-unavoidable for $k\leq 6,$ so it remains to verify the theorem statement for $k\geq 7.$  We proceed by induction on $k.$  Suppose for some $k\geq 6$ that $\formula{k}$ is $3$-unavoidable.  Suppose towards a contradiction that $w\in\{0,1,2\}^{\omega}$ is a recurrent word that avoids $\formula{k+1}.$

By the induction hypothesis, we know that $\formula{k}$ occurs in $w,$ say through morphism $h.$  We claim that $h$ must be $1$-uniform.  First suppose that $|h(x)|>1.$  Let $h(x)=va,$ where $v\neq \varepsilon$ and $|a|=1.$  Then $\formula{k+1}$ occurs in $w$ through $g$ defined by
\begin{align*}
x&\mapsto v, \mbox{ and }\\
y_i&\mapsto \begin{cases}
a &\mbox{ if } i=1\\
h(y_{i-1}) &\mbox{ if } i>1.
\end{cases}
\end{align*}
Now suppose that $|h(y_j)|>1$ for some $j\in\{1,\dots,k\}.$  Let $h(y_j)=va,$ where $v\neq \varepsilon$ and $|a|=1.$  Since $h(y_j)$ is reversible in $w$ by the fact that $h$ is an occurrence of $\formula{k}$, $v$ must also be reversible in $w$, and we see that $\formula{k+1}$ occurs in $w$ through $g$ defined by
\begin{align*}
x&\mapsto h(x), \mbox{ and }\\
y_i&\mapsto\begin{cases}
h(y_i) &\mbox{ if } i<j\\
v &\mbox{ if } i=j\\
a &\mbox{ if } i=j+1\\
h(y_{i-1}) &\mbox{ if } i>j+1.
\end{cases}
\end{align*}
Thus $h$ must be $1$-uniform as claimed.

Now take an instance of the fragment $xy_1\dots y_kx$ under some occurrence of $\formula{k}$ in $w$ through a $1$-uniform morphism.  Without loss of generality we may assume that $x\mapsto 0,$ so an instance of $xy_1\dots y_k x$ has the form
\[
0a_1\dots a_k 0,
\]
where $a_i\in\{0,1,2\}.$  First note that this factor cannot be preceded or followed by a $0$ in $w$ as this gives an obvious occurrence of $\formula{k+1}.$  Thus this factor appears internally in $w$ as either $10a_1\dots a_k01$ or $10a_1\dots a_k 02$ (up to relabelling of letters).

Consider the former possibility $10a_1\dots a_k01.$  If $a_1=0$ or $a_1=1,$ then the factor $0a_1$ is reversible, so $\formula{k+1}$ occurs in $w$ as follows:
\[
\underbrace{1}_{\overset{\upmapsto}{x}}
\underbrace{0a_1}_{\overset{\upmapsto}{y_1}}
\underbrace{a_2}_{\overset{\upmapsto}{y_2}}\dots 
\underbrace{a_k}_{\overset{\upmapsto}{y_k}} 
\underbrace{0}_{\overset{\upmapsto}{y_{k+1}}} 
\underbrace{1}_{\overset{\upmapsto}{x}}
\]
Thus we may assume that $a_1=2.$  By a symmetric argument, we may assume that $a_k=2.$  However, then $\formula{k+1}$ occurs in $w$ as follows:
\[
\underbrace{1}_{\overset{\upmapsto}{x}}
\underbrace{02}_{\overset{\upmapsto}{y_1}}
\underbrace{a_2}_{\overset{\upmapsto}{y_2}}\dots 
\underbrace{a_{k-1}}_{\overset{\upmapsto}{y_{k-1}}}
\underbrace{2}_{\overset{\upmapsto}{y_k}} 
\underbrace{0}_{\overset{\upmapsto}{y_{k+1}}} 
\underbrace{1}_{\overset{\upmapsto}{x}}
\]

We may now assume that $0a_1\dots a_k0$ appears internally as $10a_1\dots a_k 02.$  In fact, we can demonstrate that it must appear interally as $10a_1\dots a_k02^{t}1$ for some $t\geq 1.$  Otherwise, $\formula{k+1}$ occurs in $w$ as follows:
\[
\underbrace{0}_{\overset{\upmapsto}{x}}
\underbrace{a_1}_{\overset{\upmapsto}{y_1}}
\underbrace{a_2}_{\overset{\upmapsto}{y_2}}\dots 
\underbrace{a_k}_{\overset{\upmapsto}{y_{k}}}
\underbrace{02^t}_{\overset{\upmapsto}{y_{k+1}}}
\underbrace{0}_{\overset{\upmapsto}{x}}.
\]

Now to reach a contradiction it suffices to show that $0a_1a_2\dots a_k0$ can be factored as $b_1b_2\dots b_k,$ where $b_i\neq \varepsilon$ and is reversible for each $i\in\{1,\dots,k\}.$  If this is the case then $\formula{k}$ occurs in $10a_1\dots a_k02^{t}1$ as follows:
\[
\underbrace{1}_{\overset{\upmapsto}{x}}
\underbrace{b_1}_{\overset{\upmapsto}{y_1}}
\underbrace{b_2}_{\overset{\upmapsto}{y_2}}\dots 
\underbrace{b_k}_{\overset{\upmapsto}{y_{k}}}
\underbrace{2^t}_{\overset{\upmapsto}{y_{k+1}}}
\underbrace{1}_{\overset{\upmapsto}{x}}.
\]

Clearly if $0a_1\dots a_k0$ contains a length $3$ reversible factor then it can be factored into $k$ reversible factors.  Hence we may assume that $0a_1\dots a_k0$ contains no factors of the form $c^3,$ $cdc,$ or $cddc$ for $c,d\in\{0,1,2\}$ (the first two factors are length $3$ reversible factors while in the third factor $cdd$ is reversible).  

\begin{description}
\item[Case I:] $0a_1\dots a_k0$ contains at least two squares of length $2$

Clearly we can factor $0a_1\dots a_k0$ into $k$ reversible factors by choosing any two length $2$ squares.  

\item[Case II:] $0a_1\dots a_k$ contains no squares of length $2$ 

By the assumption that $0a_1\dots a_k0$ has no factors of the form $cdc$ with $c,d\in\{0,1,2\},$ we necessarily have $k\equiv 2\pmod{3}$ and $0a_1\dots a_k0$ must have the cyclic form $021021\dots 0$.  In particular, we must have $k\equiv 2\pmod{3}.$  Along with the supposition that $k\geq 4$ this implies that $k\geq 6.$  By Lemma \ref{3CyclicUnavoidable}, we know that the entire $\omega$-word $w$ cannot have cyclic form, so that at least one of the factors $02,$ $21,$ and $10$ is reversible in $w.$   Each factor $02,$ $21,$ and $10$ appears at least twice in $0a_1\dots a_k0$ (recall $k\geq 6$), meaning that there will be at least two nonoverlapping reversible factors of length $2$ in $0a_1\dots a_k0,$ allowing us to factor it into $k$ reversible factors.

\item[Case III:] $0a_1\dots a_k0$ contains exactly one square of length $2$

By the assumption that $0a_1\dots a_k0$ has no factors of the form $cdc$ or $cddc$ with $c,d\in\{0,1,2\},$ we see that $0a_1\dots a_k0$ must have the cyclic form $0^{p_1}2^{p_2}1^{p_3}0^{p_4}2^{p_5}1^{p_6}\dots 0^{p_{k+1}}$ where $p_i=2$ for some $i\in\{1,\dots,k+1\}$ and $p_j=1$ for all $j\neq i.$  In particular, we must have $k\equiv 0\pmod{3}$.  Again, by Lemma \ref{3CyclicUnavoidable} we know that at least one of the factors $02$, $21$, and $10$ must be reversible in $w.$  It is easy to verify that each of these factors appears at least once in $0^{p_1}2^{p_2}1^{p_3}0^{p_4}\dots 0^{p_{k+1}}$ without overlapping the square of length $2$ (again, recall $k\geq 6$).  Thus we can factor $0a_1\dots a_k0$ into $k$ reversible factors.
\end{description}

We have shown that $\formula{k+1}$ occurs in $w,$ a contradiction.  By induction, we conclude that $\formula{k}$ is $3$-unavoidable for all $k\geq 1.$
\end{proof}

While we conjecture that $\formula{k}$ is actually $4$-unavoidable for all $k\geq 2,$ the proof technique used for Theorem \ref{3unavoidable} does not seem tractable on a $4$-letter alphabet.  It appears that a different technique will be necessary in order to show that this conjecture holds.

\section{Conclusion}

The family $\{\formula{k}\colon\ k\geq 1\}$ is the first infinite family of avoidable formulas (with reversal) that we know of whose members are all $3$-unavoidable.  While we have shown that $\avoidindex(\formula{1})=4$ and $\avoidindex(\formula{k})=5$ for $k\in\{2,3,6\},$ we have only demonstrated bounds on the avoidability index of the remaining formulas.  It would be nice to know the exact avoidability index of $\formula{k}$ for all $k\geq 2$; we suspect that it is $5$.

After discovering such simple formulas with reversal of avoidability index $5$, it seems plausible to us that there are avoidable formulas with reversal of avoidability index $6$ which can be found.  This would be especially interesting as there are currently no known formulas (with or without reversal) of avoidability index $6.$

\providecommand{\MR}{\relax\ifhmode\unskip\space\fi MR }
\providecommand{\MRhref}[2]{%
  \href{http://www.ams.org/mathscinet-getitem?mr=#1}{#2}
}
\providecommand{\href}[2]{#2}

\appendix

\section{Walnut commands for Theorem \ref{Formula2Theorem}}\label{WalnutAppendix}

In order to show that the word $w=\rho^\infty(2)$ avoids factors of the form $x'\alpha_1x''$ and $x'\alpha_1\alpha_2x''$ using \textit{Walnut}, we first create a text file defining the automaton with output corresponding to $w$.  The text is as follows:

\begin{verbbox}
msd_2
0 2
0 -> 0
1 -> 1
1 1
0 -> 0
1 -> 0
\end{verbbox}

\[
\theverbbox
\]

We save this text to a file called ``{\tt w.txt}'' in the directory ``{\tt /Walnut/Word Automata Library/}''.  Next, we construct two automata for use in later commands: an automaton that accepts binary numbers equivalent to $3$ modulo $5,$ and another that accepts binary numbers equivalent to $4$ modulo $5.$  These automata are constructed by the following commands, respectively.

\begin{verbbox}
def is3mod5base2 "?msd_2 Ek n=5*k+3";
def is4mod5base2 "?msd_2 Ek n=5*k+4";
\end{verbbox}

\[
\theverbbox
\]

Now we run the following command:
\begin{verbbox}
eval Case1 "?msd_2 $is4mod5base2(n) 
    & (Ei w[i]>=w[i+n+1] & w[i+n-1]<=w[i+2*n] 
    & Ak (k>0 & k<n-1) => w[i+k]=w[i+n+k+1]
    & w[i+n]=w[0])";
\end{verbbox}

\[
\theverbbox
\]

This command returns an automaton that accepts all natural numbers $n$ for which $w$ has a factor of the form $x'\alpha_1x''$ satisfying the conditions of Lemma \ref{GeneralLemma} (with $|x'|=|x''|=n)$.  The automaton with a single, nonaccepting state is returned, meaning that $w$ has no factors of this form.

Finally, we run the following command:

\begin{verbbox}
eval Case2 "?msd_2 $is3mod5base2(n) 
    & (Ei w[i]>=w[i+n+2] & w[i+n-1]<=w[i+2*n+1]
    & Ak (k>0 & k<n-1) => w[i+k]=w[i+n+k+2])";
\end{verbbox}

\[
\theverbbox
\]

This command returns an automaton that accepts all natural numbers $n$ for which $w$ has a factor of the form $x'\alpha_1\alpha_2x''$ satisfying the conditions of Lemma \ref{GeneralLemma} (with $|x'|=|x''|=n)$.  As above, the automaton with a single, nonaccepting state is returned, meaning that $w$ has no factors of this form.


\begin{thebibliography}{1}

\bibitem{Index4}
K.~A. Baker, G.~F. Mc{N}ulty, and W.~Taylor, \emph{Growth problems for
  avoidable words}, Theoret. Comput. Sci. \textbf{69} (1989) 319--345.

\bibitem{Brandenburg}
F.-J. Brandenburg, \emph{Uniformly growing $k$-th power-free homomorphisms},
  Theoret. Comput. Sci. \textbf{23}(1) (1983) 69--82.

\bibitem{CassaigneThesis}
J.~Cassaigne, \emph{Motifs \'evitables et r\'egularit\'e dans les mots}, Ph. D.
  thesis, Universit\'e Paris VI (1994).

\bibitem{ClarkThesis}
R.~J. Clark, \emph{Avoidable formulas in combinatorics on words}, Ph. D. thesis,
  University of California, Los Angeles (2001).

\bibitem{BinaryPatternsReversal}
J.~D. Currie and P.~Lafrance, \emph{Avoidability index for binary patterns with
  reversal}, Electron. J. Combin. \textbf{23}(1) (2016) P1.36.

\bibitem{CurrieRampersad2015}
J.~D. Currie and N.~Rampersad, \emph{Binary words avoiding {$xx^Rx$} and
  strongly unimodal sequences}, J. Integer Seq. \textbf{18}(15.10.3) (2015) 1--7.

\bibitem{CurrieRampersad2016}
J.~D. Currie and N.~Rampersad, \emph{Growth rate of binary words avoiding $xxx^r$}, Theoret. Comput.
  Sci. \textbf{609} (2016) 456--468.

\bibitem{Walnut}
H.~Mousavi, \emph{Automatic theorem proving in {W}alnut},	arXiv:1603.06017 [cs.FL] (2016).

\bibitem{RichardGrimm2004}
C.~Richard and U.~Grimm, \emph{On the entropy and letter frequencies of ternary
  square-free words}, Electron. J. Combin. \textbf{11}(1) (2004) R14.

\end{thebibliography}
\end{document}